\documentclass{amsart}

\usepackage{amsmath, amssymb, amsthm, xcolor}
\usepackage[all]{xy}
\usepackage{ulem}

\newtheorem{thm}{Theorem}
\newtheorem{lemma}[thm]{Lemma}
\newtheorem{notation}[thm]{Notation}
\newtheorem{prop}[thm]{Proposition}
\newtheorem{cor}[thm]{Corollary}

\newtheorem{ques}[thm]{Question}
\theoremstyle{definition}
\newtheorem{defn}[thm]{Definition}

\newtheorem{remark}[thm]{Remark}
\numberwithin{thm}{section}

\DeclareMathOperator{\depth}{depth}

\DeclareMathOperator{\Tor}{Tor}

\newcommand{\m}{\mathfrak{m}}

\newcommand{\n}{\mathfrak{n}}

\renewcommand{\phi}{\varphi}

\DeclareMathOperator{\height}{ht}

\DeclareMathOperator{\Ht}{ht}

\renewcommand{\to}{\longrightarrow}

\newcommand{\Aa}{A_{\infty,0}}
\newcommand{\Ab}{A_{\infty}}

\newcommand{\Ba}{B\otimes_AA_{\infty}}

\newcommand{\Am}{\mathcal A}
\newcommand{\Bm}{\mathcal B}

\newcommand{\etale}{{\'e}tale }

\title{Extended Plus Closure in Complete Local Rings}
\author{Raymond C Heitmann and Linquan Ma}
\date{\today}

\thanks {The second author was supported in part by an NSF Grant \#1836867/1600198 and an NSF CAREER Grant DMS \#1252860/1501102 when preparing this article. Both authors would like to thank Mel Hochster and Kazuma Shimomoto for their comments.}

\begin{document}
\maketitle

\begin{abstract}
The (full) extended plus closure was developed as a replacement for tight closure in mixed characteristic rings.
Here it is shown by adapting Andr\'{e}'s perfectoid algebra techniques that, for complete local rings {that have F-finite residue fields}, this closure has the colon-capturing property.
In fact, more generally, if $R$ is a (possibly ramified) complete regular local ring of mixed characteristic {that has an F-finite residue field}, $I$ and $J$ are ideals of $R$, and the local domain $S$ is a finite $R$-module, then $(IS:J)\subseteq (I:J)S^{epf}$.
A consequence is that all ideals in regular local rings are closed, a fact which implies the validity of the direct summand conjecture and the Brian\c con-Skoda theorem in mixed characteristic.
\end{abstract}

\section{Introduction}

In \cite{H2}, the first author introduced several closure operations for mixed characteristic rings, most notably the full extended plus closure.
These closures were proposed as possibilities to play the same role in mixed characteristic as the tight closure does in characteristic $p$.
It was hoped that they might allow us to prove some of the homological conjectures that remained unresolved in mixed characteristic and indeed, the full extended plus closure did figure in the proof of the direct summand conjecture in dimension 3 \cite{H3}.

Of course, the direct summand conjecture has been resolved \cite{A2} without the use of this closure operation.
However, it still remains of interest to note to what extent the full extended plus closure fills the void left open by the absence of tight closure in mixed characteristic.
In this paper, we note a number of properties that were shown in \cite{H2} and add several others.
In particular, we deal with two fundamental questions that were left open in \cite{H2}.
We shall see that the full extended plus closure has the colon-capturing property in complete local rings {that have F-finite residue fields} and that ideals in all regular local rings are closed.
The latter fact is sufficient to imply the direct summand conjecture.
This is not truly an alternate proof as our work makes use of Andr\'{e}'s notation and techniques.
For the most part, this article is about the full extended plus closure in complete local rings.
Since we do not know whether or not completing, closing, and contracting occasionally gives a larger closure, only Theorem~\ref{RLR} gives new information about rings which are not complete.

The primary result in this article is that, in a sense, the full extended plus closure captures obstructions to flatness.
{Throughout this article, when we refer to a set as \textit{parameters} in a local ring, we mean they are a subset of a full system of parameters.}
Our principal theorems are:

\begin{thm} [Corollary~\ref{CC}]
Let $S$ be a complete local domain of mixed characteristic {that has an F-finite residue field} and let $y_1,\ldots,y_n$ be parameters in $S$.
Then $((y_1,\ldots,y_{n-1})S:y_n)\subseteq (y_1,\ldots,y_{n-1})S^{epf}$.
\end{thm}

\begin{thm} [Theorem~\ref{new main}]
Let $R$ be a complete regular local ring of mixed characteristic {that has an F-finite residue field} and let the integral domain $S$ be a finite extension of $R$.
If $I$ is an ideal of $R$ and $J\subseteq R$, then $(IS:J)\subseteq (I:J)S^{epf}$.
\end{thm}

{We would also like to state an equivalent formulation of Theorem 1.2, divorced from the language of extended plus closure.
Before doing so however, we should explain some terminology.
First, if $R$ is any integral domain, $R^+$, refers to the integral closure of $R$ in some algebraic closure of its quotient field.
$R^+$ is commonly referred to as the absolute integral closure.
Also, throughout this article, we will often use the phrase ``for all $\epsilon>0$".
It is assumed that $\epsilon$ is a positive rational number as $c^\epsilon$ is not otherwise defined.
Readers who routinely work in almost mathematics may harmlessly assume $\epsilon=\frac 1{p^n}$.}

\begin{thm} [Theorem~\ref{Tor}]
Let $R$ be a complete regular local ring of mixed characteristic {that has an F-finite residue field} and let the integral domain $S$ be a finite extension of $R$.
Let $M$ be an $R$-module.
Then there exists a nonzero $c\in R$ such that the natural map $c^{\epsilon}\Tor_i^R(S,M)/p^N\to \Tor_i^R(S^+,M)/p^N$ is zero for all $i\geq 1$ and for every $N,\epsilon>0$. Consequently, for every $\alpha\in \Tor_i^R(R^+,M)/p^N$, there exists a nonzero $c\in R$ such that $c^{\epsilon}\alpha=0$ for every $\epsilon$.
\end{thm}

Our results also have implications regarding a question about local cohomology posed by P. Roberts, A. Singh, and V. Srinivas.
Suppose $(R,\m)$ is an excellent local domain of dimension $n$.
It is well known that $H_\m^i(R)=0$ for $i>n$ and that $H_\m^n(R)\neq 0$.
We also know that $H_\m^i(R)=0$ if $i<\depth R$ and so the lower local cohomology modules ($i<n$) ``measure" how close $R$ is to being Cohen-Macaulay.
M. Hochster and C. Huneke demonstrated \cite{HH4} that, in the equal characteristic $p$ case, all lower local cohomology modules $H_\m^i(R^+)$ vanish.
In a different setting, one of the implications of the proof of the mixed characteristic direct summand conjecture in dimension three \cite{H3} was the observation that $H_\m^2(R^+)$ is almost zero in the sense that it can be annihilated by arbitrarily small powers of $p$.
This hinted at an ``almost" form of the Hochster-Huneke result in the mixed characteristic case.
Motivated by this, P. Roberts, A. Singh, and V. Srinivas posed the following characteristic free question \cite[Question 1.5]{RSS}:
\begin{ques}\label{Q}
Let $(R,\m)$ be a complete local domain.
For $i< \dim R$, is the image of the natural map  $H_\m^i(R)\to H_\m^i(R^+)$ almost zero?
\end{ques}
Of course, as noted above, the question has an affirmative answer in the equicharacteristic $p$ case.
In their article, Roberts et al. considered a number of families of examples and produced considerable evidence in support of a positive answer in the equicharacteristic zero case.
In a somewhat different direction, H. Brenner and A. St\"abler \cite{BS} also explored almost zero cohomology in equicharacteristic zero.
However, until the present article, there has been no further progress in the mixed characteristic case.
Here we shall demonstrate that Question~\ref{Q} has an affirmative answer in the mixed characteristic case.

\begin{thm} [Theorem~\ref{Cohom}]
Let $(R,\m)$ be a complete local domain of mixed characteristic {that has an F-finite residue field}.
Then there exists a nonzero $c\in R$ such that the natural map  $c^\epsilon H_{\m }^i(R)\to H_{\m }^i(R^+)$ is zero for all $i<\dim R$ and all $\epsilon>0$.
\end{thm}

As a consequence, $H_{\m}^i(R^+)$ is almost zero if almost zero is interpreted in the sense of Roberts, Singh, and Srinivas.

This paper is organized as follows. In Section 2, we will give a brief introduction to extended plus closure.
The results discussed here are generally known although Proposition~\ref{TC} is actually new.
We will also give a short introduction to our terminology and conventions.
Section 3 will be devoted to our new results.

\section{basics}

\begin{notation}
If $I$ is an ideal of $R$ and $J\subseteq R$, $(I:J)=\{r\in R\mid rJ\subseteq I\}$.
The ring involved will always be clear from the ideal and there will be no need for the more cumbersome notation $(I:_RJ)$.
\end{notation}

 We recall that a ring $R$ of equal characteristic $p>0$ is called {\it F-finite} if the Frobenius map $F$: $R\to R$ makes $R$ into a module-finite $R$-algebra. In particular, a field $k$ of characteristic $p>0$ is F-finite if and only if $k^{1/p}$ is a finite dimensional vector space over $k$.

In \cite{H2}, the first author introduced four closure operations for mixed characteristic rings.
Subsequently, only one of these, the full extended plus closure, has received any attention.
That will also be the case in this article. In fact, while we shall preserve the notation $I^{epf}$, the $f$ for "full", we shall mostly refer to the closure as the \textit{extended plus closure}. It should be noted that the original definition and early results allow $R$ to be either mixed characteristic or characteristic $p$. However, the characteristic $p$ case is less interesting as the full extended plus closure trivially contains the tight closure.

{
\begin{notation}
As this article is concerned with the extended plus closure, we shall only be concerned with rings for which the extended plus closure makes sense.
Throughout (excepting only Theorem~\ref{L}), $p$ will be a specific prime integer and $R$ will be a ring such that the image of the prime integer $p$ is in the Jacobson radical of $R$, which we denote as $J(R)$.
If the image of $p$ is nonzero, we say that $R$ is a ring of mixed characteristic $p$.
\end{notation}}

\begin{defn}
Let $R$ be an integral domain {such that the image of the prime integer $p$ is in $J({R})$, and let }$I$ an ideal of $R$.
Then an element $x\in R$ is said to be in the (full) extended plus closure of $I$ (designated $x\in I^{epf}$) provided there exists $c\in R$ such that for every $\epsilon>0$ and every positive integer $N$, $c^{\epsilon}x\in (I,p^N)R^+$.
\end{defn}

\begin{remark}
The extended plus closure is actually {also} defined for all Noetherian rings {such that the image of the prime integer $p$ is in $J({R})$.}
Just as with tight closure, one may compute the extended plus closure by computing it modulo each minimal prime ideal and then taking the intersection of the liftings back to the original ring.
The more general setting does not hold great interest and in this article, we will focus on the case of integral domains.
\end{remark}

We first shall highlight some of the basic results from \cite{H2}. First we see that it behaves as a reasonable closure operation.

\begin{prop}\label{elem}\cite[Proposition 2.1]{H2}
Let $I,J$ be ideals in a Noetherian ring $R$.
\begin{enumerate}
  \item $I^{epf}$ is an ideal.
  \item  $I\subseteq I^{epf}$.
  \item   If $I\subseteq J$, $I^{epf}\subseteq J^{epf}$.
    \item $(I^{epf})^{epf}=I^{epf}$.
  \item    $(I^{epf}J^{epf})^{epf} =(IJ)^{epf}$.
  \item   $(I:J)^{epf}\subseteq (I^{epf}:J)$.
  In particular, if $I$ is closed, then $(I:J)$ is closed.
\end{enumerate}
\end{prop}

\begin{remark}
It is known that if $R$ is local and $S=R^h$, the henselization of $R$, then $IS^{epf}\cap R=I^{epf}$ \cite[Theorem 3.4] {H2}. However, this result is unfortunately not known if we replace $S$ by the completion $\widehat{R}$ of $R$, even when $R$ is an excellent domain. For this reason our main theorems on extended plus closure do not obviously generalize to non-complete rings.
\end{remark}

Finally, we have the two most interesting theorems from the earlier work - two theorems which match consequential theorems from tight closure theory.
First full extended plus closure gives a Brian\c con-Skoda type theorem.

\begin{thm}  \cite[Theorem 4.2]{H2}
Let $R$ be {an integral domain}, let $I$ be an ideal of $R$ generated by $n$ elements, and suppose $y$ is in the integral closure of $I^{n+d}$ for some integer $d\geq 0$.
Then $y\in (I^{d+1})^{epf}$.
\end{thm}

We note that this Brian\c con-Skoda type theorem is a simple consequence of  \cite[Theorem 2.13]{H1}.
We include that result next as it serves as a lemma we will need in our proof of Lemma~\ref{new complete} and Theorem~\ref{L}.

\begin{prop}\label{BS}
Let $R$ be an integral domain and $I=(x_1,\ldots,x_n)R$ an ideal of $R$.
Suppose $p\in\sqrt{(x_1,x_2)R}$ and $z\in \overline {I^{n+k}}$ with $k\geq 0$.
Then there exists an integral extension $S$ of $R$ with $z\in I^{k+1}S$.
\end{prop}

Recall that the classical Brian\c con-Skoda theorem says that in a regular ring, if an ideal $I$ is generated by $n$ elements, then the integral closure of $I^{n+d}$ is contained in $I^{d+1}$ for every integer $d$ \cite{LS}. We can recover this result by proving that in a regular local ring, every ideal is closed under the extended plus closure; see Theorem \ref{RLR}. We give two proofs of this theorem: it can be obtained from our earlier work on the vanishing conjecture for maps of Tor \cite{HM}, and it also follows from our main result in this article and the following result in \cite{H2}:

\begin{thm} \label{4.4} \cite[Theorem 4.4]{H2}
Let $R$ be a regular local ring with $p$ in the maximal ideal.
If the full extended plus closure has the colon-capturing property for {integral domains which are} finite extensions of complete regular local rings, then $I^{epf}=I$ for every ideal $I$ of $R$.
\end{thm}

\begin{remark}
The actual statement of this result in \cite{H2} presumes that the colon-capturing property holds for finite extensions of all regular local rings.
{However, it is clear from the proof that it suffices to consider extensions which are domains.
Also,} the statement of \cite[Lemma 4.3]{H2} makes it clear that one may restrict to complete regular local rings.
Then, by faithful flatness, the fact that $I^{epf}=I$ holds in complete regular local rings implies it holds in all regular local rings.
\end{remark}

One may observe that these results are all well known for tight closure.
In fact, an alternate definition of tight closure closely resembles that of extended plus closure.
Instead of requiring $c^{\epsilon}x\in IR^+$, the tight closure requires
$c^{\epsilon}x\in IR^F$  where $R^F$ is the integral closure of $R$ in the largest purely inseparable extension of its quotient field.
Superficially, this means the tight closure could be smaller.
However, Hochster and Huneke \cite{HH3} introduced the dagger closure (which must contain the extended plus closure) and showed \cite[Theorem 3.1]{HH3} that, for characteristic $p$ complete local rings, the tight closure and dagger closure were equal.
Thus

\begin{prop}\label{TC}
If $R$ is a complete local ring of equal characteristic $p$, the extended plus closure coincides with tight closure for ideals of $R$.
\end{prop}

\subsection{Perfectoid algebras} We will freely use the language of perfectoid spaces \cite{S} and almost mathematics \cite{GR}. In this paper we will always work in the following situation: for a perfect field $k$ of characteristic $p>0$, we let $W(k)$ be the ring of Witt vectors with coefficients in $k$. Let $K^\circ$ be the $p$-adic completion of $W(k)[p^{1/p^\infty}]$ and $K=K^\circ[1/p]$. Then $K$ is a {\it perfectoid field} in the sense of \cite{S} with $K^\circ\subseteq K$ its ring of integers.

A {\it perfectoid $K$-algebra} is a Banach $K$-algebra $R$ such that the set of powerbounded elements $R^\circ\subseteq R$ is bounded and the Frobenius is surjective on $R^\circ/p$. A $K^\circ$-algebra $S$ is called {\it integral perfectoid} if it is $p$-adically complete, $p$-torsion free, satisfies $S=S_*$, and the Frobenius induces an isomorphism $S/p^{1/p}\to S/p$.  Here $S_*=\{x\in S[1/p] \hspace{0.3em} | \hspace{0.3em} p^{1/p^k}\cdot x\in S \text{ for all } k \}$, so $S$ is almost isomorphic to $S_*$ with respect to $(p^{1/p^\infty})$. In practice, we will often ignore this distinction since one can always pass to $S_*$ without affecting the issue. These two categories are equivalent to each other \cite[Theorem 5.2]{S} via the functors $R\to R^\circ$ and $S\to S[1/p]$.

Almost mathematics in this article will be measured with respect to a flat ideal $(c^{1/p^\infty})$. There are two cases that we will use: $c=p$ or $c=pg$ for some nonzero element $g$. This will be stated explicitly throughout.

\section{Main Results}

For $R$ a Noetherian integral domain of mixed characteristic $p$, we recall that $R^+$ denotes the absolute integral closure of $R$. We let $R^{++}$ denote the $p$-adic completion of $R^+$. For a nonzero element $c\in R$, we will use $R^c$ to denote the integral closure of $R^{++}$ in $R^{++}[c^{-1}]$.

\begin{lemma}\label{complete}
Let $R$ be an integral domain of mixed characteristic $p$ and $I$ an ideal of $R$. For $x\in R$, if there exists $c\in R$ such that for every $\epsilon>0$ and every positive integer $N$, $c^{\epsilon}x\in (I,p^N)R^{++}$, then $x\in I^{epf}$.
\end{lemma}

\begin{proof}
We need only show that if $c^{\epsilon}x\in (I,p^N)R^{++}$, then $c^{\epsilon}x\in (I,p^N)R^+$.
However, if $c^{\epsilon}x\in (I,p^N)R^{++}$, we have $c^{\epsilon}x=\sum_{i=1}^sy_iu_i+p^Nu_0$ where each $y_i\in I$ and each $u_i$ is the limit of a Cauchy sequence $\{u_{ij}\}$.
It follows that $\{c^{\epsilon}x-\sum_{i=1}^sy_iu_{ij}-p^Nu_{0j}\}$ is a Cauchy sequence which converges to zero.
Thus, for sufficiently large $j$, we have $c^{\epsilon}x-\sum_{i=1}^sy_iu_{ij}-p^Nu_{0j}\in p^NR^+$.
Then $c^{\epsilon}x=\sum_{i=1}^sy_iu_{ij}+p^N(u_{0j}+v)\in (I,p^N)R^+$.
\end{proof}

\begin{lemma}\label{intersection}
Let $R$ be a universally catenarian Noetherian integral domain of mixed characteristic $p$ and suppose $(p,g)R$ is a height two ideal. Then we have $\cap_n(g, p^n)R^{++}=gR^{++}$.
\end{lemma}
\begin{proof}
Pick $z\in \cap_n(g, p^n)R^{++}$;  we can write $z=a_1g+b_1p$ with $a_1\in R^+$ and $b_1\in R^{++}$.
Now $b_1p\in \cap_n(g, p^n)R^{++}$ and so we can write $b_1p=a_2g+b_2p^2$ with $a_2\in R^+$ and $b_2\in R^{++}$.
But then $b_2p^2\in \cap_n(g, p^n)R^{++}$, so we can repeat the above process to write $b_2p^2=a_3g+b_3p^3$ and keep going: $$z=a_1g+b_1p=a_1g+a_2g+b_2p^2=a_1g+a_2g+a_3g+b_3p^3=\cdots.$$ Note that we have $a_ig\in p^{i-1}R^{++}$ for all $i$.  Thus $a_ig=0$ as an element in $R^+/p^{i-1}=R^{++}/p^{i-1}$. Since $R$ is universally catenary, $(p,g)$ is a height two ideal in any integral extension of $R$. Since $R^+$ is a direct limit of normal integral extensions of $R$, $g$ is a nonzerodivisor on $R^+/p^{i-1}$, and so $a_i\in p^{i-1}R^+$. Therefore, in $R^{++}$ we have $$z=a_1g+a_2g+a_3g+\cdots=g(a_1+a_2+a_3+\cdots)\in gR^{++}.$$
This finishes the proof.
\end{proof}

\begin{lemma}\label{new complete}
Let $R$ be a universally catenarian Noetherian integral domain of mixed characteristic $p$ of dimension at least two and $I$ an ideal of $R$. For $x\in R$, if there exists a nonzero $c\in R$ such that for every $\epsilon>0$ and every positive integer $N$, $c^{\epsilon}x\in (I,p^N)R^c$, then $x\in I^{epf}$.
In fact, the injection $R^{++}\to R^c$ is an almost isomorphism with respect to $(c^{1/p^{\infty}})$.
\end{lemma}
\begin{proof}
We can find an element $g\in R$ such that $pg\in \sqrt{cR}$ and $(p,g)R$ is a height two ideal; likewise so is $(p,g)R'$ for every integral extension $R'$ of $R$ (since $R$ is universally catenary).
The hypothesis of the lemma clearly holds with $pg$ in place of $c$ (note that $R^{pg}$ is larger than $R^c$) and so we may assume $c=pg$. We first show that $pg$ is a nonzerodivisor on $R^{++}$, it is enough to prove that $g$ is a nonzerodivisor. But if $gx=0$ in $R^{++}$ and $x\neq 0$, then there exists $a\in R^+$ such that $ x-a\in p^nR^+$ but $a\notin p^nR^+$, and $ga\in p^nR^{++}$. It follows that $ga\in p^nR^+$, contradicting the fact that $p,g$ is regular in $R^+$.
Now this lemma will follow immediately from Lemma~\ref{complete} if we can show that the injection $R^{++}\to R^c$ is an almost isomorphism with respect to $(c^{1/p^{\infty}})$.

Suppose $u\in R^c$.
Then $u=(pg)^{-d}v$ for $v\in R^{++}$ and some integer $d$.
As $R^{++}$ is just the $p$-adic completion of $R^+$, we may write $v=w_1+p^dw_2$
with $w_1\in R^+$ and $w_2\in R^{++}$.
It follows that $\frac {w_1}{p^d}=g^du-w_2$ is integral over $R^{++}$.
{This gives us an equation $(\frac {w_1}{p^d})^s+b_1(\frac {w_1}{p^d})^{s-1}+\cdots +b_s=0$ with each $b_i\in R^{++}$.
As each $b_i\in R^++p^{(d-1)s}R^{++}$, we can reduce to the case where $b_i\in R^+$ for all $i<s$.
Then $p^{ds}b_s\in R^+\cap p^{ds}R^{++}=p^{ds}R^+$ and so $b_s\in R^+$ as well.
Thus  $\frac {w_1}{p^d}$ is} integral over the integrally closed $R^+$.
So without loss of generality, we may assume $w_1=0$ and so reduce to the case
$u=g^{-d}v$.

Fix $t>0$ and let $\epsilon=1/p^t$.
To complete the proof it suffices, by Lemma~\ref{intersection}, to show $v\in (g^{d-\epsilon}, p^n)R^{++}$ for all positive integers $n$ . Now we fix an $n>0$.
We may write $v=v_1+p^{nN}v_2$ with $v_1\in R^+$, $v_2\in R^{++}$, and
$N=dp^{t}=d/\epsilon$.
As $v$ is integral over $g^dR^{++}$, $v_1$ is integral over $(g^d,p^{nN})R^{++}$. This means there exists some integer $k$ such that $(v_1, g^d, p^{nN})^{k+1}R^{++}=(v_1,g^d,p^{nN})^k(g^d,p^{nN})R^{++}$. Since both ideals contain a power of $p$, it is easy to see that $(v_1, g^d, p^{nN})^{k+1}R^{+}=(v_1,g^d,p^{nN})^k(g^d,p^{nN})R^{+}$ and thus $v_1$ is integral over $(g^d,p^{nN})R^+$.
It follows that $v_1\in \overline{(g^{\epsilon},p^n)^NR^+}$.
As $R^+$ is absolutely integrally closed, Proposition~\ref{BS} gives
$v_1\in {(g^{\epsilon},p^n)^{N-1}R^+}\subseteq (g^{d-\epsilon},p^n)R^+$. Therefore $v=v_1+p^{nN}v_2\in (g^{d-\epsilon},p^n)R^{++}$ as desired.
\end{proof}

Now suppose $R$ is a complete and unramified regular local ring of mixed characteristic {that has an F-finite residue field}, and $R\to S$ is a module-finite domain extension (so $S$ is a complete local domain). By Cohen's structure theorem $R\cong W(k)[[x_1,\dots,x_{d-1}]]$ where $W(k)$ is a coefficient ring of $R$, which is a complete and unramified DVR with an F-finite residue field $k$. Let $k^{1/p^\infty}$ denote the perfect closure of $k$, let $A_0=R\otimes_{W(k)} W(k^{1/p^\infty})$, and let $A$ be the $p$-adic completion of $A_0$. Then $A$ is $p$-adically complete and formally smooth over $W(k^{1/p^\infty})$ mod $p$. We next point out that $A$ is actually Noetherian (though we don't need this).

\begin{lemma}
With the notation as above, $A$ is a Noetherian ring.
\end{lemma}
\begin{proof}
First of all, in general, if $T$ is $p$-adically complete, then $T$ is Noetherian if and only if $T/pT$ is Noetherian.
{To see this, suppose $Q$ is an infinitely generated prime ideal of $T$.
Then, since $T/pT$ is Noetherian, $Q$ does not contain $p$.} Now $Q+pT$ is finitely generated, say by $f_1+pt_1,\dots,f_n+pt_n$.
Suppose $f\in Q\subseteq Q+pT$; we can write $f=\sum a_{1j}(f_j+pt_j)=\sum a_{1j}f_j+p\sum a_{1j}t_j$. Now $p\sum a_{1j}t_j\in Q$ but $p\notin Q$, so $\sum a_{1j}t_j\in Q\subseteq Q+pT$ and thus $\sum a_{1j}t_j=\sum a_{2j}(f_j+pt_j)=\sum a_{2j}f_j+p\sum a_{2j}t_j$. This gives $f=\sum (a_{1j}+pa_{2j})f_j+p^2\sum a_{2j}t_j$. We can repeat this process to find $a_{ij}$ and since $T$ is $p$-adically complete, we have $f=\sum (\sum_ip^{i-1}a_{ij}) f_j$. Therefore $Q=(f_1,\dots,f_n)$ is finitely generated.

Now to see $A$ is Noetherian, it is enough to show $A/pA=k^{1/p^\infty}\otimes k[[x_1,\dots,x_{d-1}]]$ is Noetherian by the previous paragraph. We prove this by induction on $d$. The case $d=1$ is clear. We now prove that every prime ideal $Q$ in $k^{1/p^\infty}\otimes k[[x_1,\dots,x_{d-1}]]$ is finitely generated. Since $k[[x_1,\dots,x_{d-1}]]\to k^{1/p^\infty}\otimes k[[x_1,\dots,x_{d-1}]]$ is integral, we have $P=Q\cap k[[x_1,\dots,x_{d-1}]]$ with $\Ht P=\Ht Q$. If $\Ht Q=0$ then $Q=0$ since $k^{1/p^\infty}\otimes k[[x_1,\dots,x_{d-1}]]$ is a domain. If $\Ht Q=\Ht P>0$, then $\bar{Q}$ is a minimal prime of $k^{1/p^\infty}\otimes (k[[x_1,\dots,x_{d-1}]]/P)$. But we have $k[[y_1,\dots,y_{d'-1}]]\to k[[x_1,\dots,x_{d-1}]]/P$ for some $d'<d$ by Cohen's structure theorem; hence $k^{1/p^\infty}\otimes (k[[x_1,\dots,x_{d-1}]]/P)$, being a finite extension of $k^{1/p^\infty}\otimes k[[y_1,\dots,y_{d'-1}]]$, is Noetherian by our induction hypothesis. Therefore $\bar{Q}$ is finitely generated and thus $Q$ is finitely generated.
\end{proof}

Returning to the development of our set-up, we next set $B=S\otimes_R A$. Since $R\to S$ is module-finite and generically \'{e}tale, there exists a nonzero $g\in R$ with $\Ht (p,g)R=2$ such that $R[(pg)^{-1}]\to S[(pg)^{-1}]$ is finite \'{e}tale. Thus $A[(pg)^{-1}]\to B[(pg)^{-1}]$ is finite \'{e}tale. Let $K^\circ$ be the $p$-adic completion of $W(k^{1/p^\infty})[p^{1/p^\infty}]$ and $K=K^\circ[1/p]$. Let $A_{\infty,0}$ be the $p$-adic completion of $A[p^{1/p^\infty},x_1^{1/p^\infty},\dots,x_{d-1}^{1/p^\infty}]$, which is an integral perfectoid $K^\circ$-algebra.\footnote{ This is the only place we need to use the F-finite hypothesis on $k$. Note that we have $A_{\infty,0}/p\cong (k^{1/p^\infty}\otimes_kk[[\bar{x}_1,\dots,\bar{x}_{d-1}]])[\bar{p}^{1/p^\infty}, \bar{x}_1^{1/p^\infty},\dots,\bar{x}_{d-1}^{1/p^\infty}]$. One can check that the Frobenius map is surjective on this ring provided that $k$ is F-finite.}
Furthermore, for $g\in R\subseteq A$ as above, we let $A_{\infty,0}\to A_\infty$ be Andr\'{e}'s construction of integral perfectoid $K^\circ$-algebras (for example see \cite[Theorem 1.4]{B}): it is the integral perfectoid $K^\circ$-algebra of functions on the Zariski closed subset of $Spa(\Aa\langle T^{1/p^\infty}\rangle[1/p], \Aa\langle T^{1/p^\infty}\rangle)$ defined by the ideal $T-g$. More explicitly, $\Ab$ can be described as the $p$-adic completion of the integral closure of $\Aa\langle T^{1/p^\infty}\rangle/(T-g)$ inside $(\Aa\langle T^{1/p^\infty}\rangle/(T-g))[1/p]$. Andr\'{e} proved that $A_\infty$ is almost faithfully flat over $A_{\infty,0}$ mod $p^m$ for every $m$ (for example, see \cite[Theorem 2.3]{B}). More importantly, Andr\'{e} proved the following remarkable result:

\begin{thm}\label{andre}\cite [Theorem 0.3.1]{A1}
Let $\mathcal A$ be a perfectoid algebra over a perfectoid field $\mathcal K$ of residue characteristic $p$.
Suppose $g\in\mathcal A^\circ$ is a nonzerodivisor and $\mathcal A$ contains a compatible system of $p$-power roots of $g$. Let $\Bm'$ be a finite \etale $\Am[\frac 1g]$-algebra. Then \begin{enumerate}
  \item There exists a larger perfectoid algebra $\Bm$ between $\Am$ and $\Bm'$ such that the inclusion $\Am\to\Bm$ is continuous.
  We have $\Bm[\frac 1g]=\Bm'$ and $\Bm^\circ$ is contained in the integral closure of $g^{-1/p^{\infty}}\Am^\circ$ in $\Bm'$.
 {Moreover, the inclusion of $\Bm^\circ$ in the integral closure of $g^{-1/p^{\infty}}\Am^\circ$ in $\Bm'$ is a $(pg)^{1/p^\infty}$-almost isomorphism.}
  \item Let $\Bm$ be as in part (1).
  For every m, $\Bm^\circ/p^m$ is $(pg)^{1/p^\infty}$-almost finite \etale over $\Am^\circ/p^m$.
\end{enumerate}
\end{thm}

\begin{remark}
\label{Bo}
It is easy to see that $\Bm^\circ$ is $(pg)^{1/p^\infty}$-almost isomorphic to the integral closure of $\Am^\circ$ in $\Bm'$: clearly this integral closure is $(pg)^{1/p^\infty}$-almost contained in $\Bm^\circ$. Now suppose we have an equation $$y^n+a_1y^{n-1}+\cdots+a_n=0$$ where $n=p^e$ and $a_1,\dots,a_n$ belongs to $g^{-1/p^{\infty}}\Am^\circ$. Then after multiplication by $g^{1/p^t}$, the above equation tells us that $g^{1/p^{t+e}}y$ is integral over $\Am^\circ$ and thus $g^{1/p^\infty}\Bm^\circ$ is $(pg)^{1/p^\infty}$-almost contained in the integral closure of $\Am^\circ$ in $\Bm'$.
\end{remark}

\begin{notation}\label{Notation} We will frequently use the following notation throughout the rest of the article. $R$ will be a complete and unramified regular local ring {that has an F-finite residue field} and $S$ a module-finite domain extension of $R$. Fix a nonzero element $g\in R$ with $\Ht (p,g)R=2$ such that $R[(pg)^{-1}]\to S[(pg)^{-1}]$ is finite \'{e}tale. We construct $A$ as in the paragraph after Lemma~\ref{new complete}, and construct $B$, $A_{\infty,0}$, $A_\infty$ as in the paragraph before Theorem~\ref{andre}. $\Am$ will be the perfectoid algebra $\Ab[p^{-1}]$ and $\Bm'$ will be $\Ba[(pg)^{-1}]$.
Let $\Bm$ and $\Bm^\circ$ be as in Theorem~\ref{andre}.
\end{notation}

The next lemma is well-known to some experts. We record it here for completeness.
\begin{lemma}\label{Bmalmostflat}
Let terminology be as in Notation~\ref{Notation}.
$\Bm^\circ$ is $(pg)^{1/p^\infty}$-almost flat over $R$.
\end{lemma}
\begin{proof}
By Theorem~\ref{andre} we know $\Bm^\circ/p^m$ is $(pg)^{1/p^\infty}$-almost flat over $A_\infty/p^m$, and hence $(pg)^{1/p^\infty}$-almost flat over $A/p^m$ by construction. Since $A$ is flat over $R$, $\Bm^\circ/p^m$ is $(pg)^{1/p^\infty}$-almost flat over $R/p^m$ for all $m$. But we also know that $\Bm^\circ$ is $p$-torsion-free and $p$-adically complete since it is integral perfectoid. Now the same argument as in \cite[Lemma 2.3]{MS} (simply replace $A$ by $R$ and $S$ by $\Bm^\circ$) shows that $\Bm^\circ$ is $(pg)^{1/p^\infty}$-almost flat over $R$.
\end{proof}


\begin{lemma}\label{almost map}
Let terminology be as in Notation~\ref{Notation}.
There exists a $(pg)^{1/p^\infty}$-almost map from $\Bm^\circ$ to $S^{pg}=R^{pg}$.
\end{lemma}
\begin{proof}
As $W(k^{1/p^\infty})$ is, up to $p$-adic completion, integral over $W(k)$, we get an embedding $W(k^{1/p^\infty})\to R^{++}$ and hence a map $A \to R^{++}$ (recall that $A$ is the $p$-adic completion of $R\otimes_{W(k)} W(k^{1/p^\infty})$). We next fix an embedding $$A[p^{1/p^{\infty}},x_1^{1/p^{\infty}},\ldots,x_{d-1}^{1/p^{\infty}}]\to R^{++}.$$ This in turn induces an embedding of the $p$-adic completions of the rings, i.e., an embedding $\alpha:\Aa\to R^{++}=S^{++}$. We next extend this to a homomorphism $\beta:\Aa\langle T^{1/p^\infty}\rangle \to S^{++}$ which sends $T^{1/p^k}$ to $g^{1/p^k}$.
In particular, since $\beta(T-g)=0$,
we get an induced homomorphism $\Aa\langle T^{1/p^\infty}\rangle/(T-g)\to S^{++}$. Finally, since $S^{++}$ is integral perfectoid, it is (at least $(p^{1/p^\infty})$-almost) isomorphic to the integral closure of $S^{++}$ in $S^{++}[1/p]$. Hence by the description of $A_\infty$, we get a  $(p^{1/p^\infty})$-almost map $\gamma: \Ab\to S^{++}$.
This induces $1\otimes \gamma:S\otimes_R\Ab\to S^{++}$. But $S\otimes_R\Ab=S\otimes_R(A\otimes_A\Ab)=\Ba$.
So this map extends to an almost map $\delta$ from the integral closure of $\Ba$ in $\Ba[(pg)^{-1}]$ to $S^{pg}$.
As $\Bm^\circ$ is $(pg)^{1/p^\infty}$-almost isomorphic to the integral closure of $A_\infty$ in $\Ba[(pg)^{-1}]$ by Remark~\ref{Bo}, we get a $(pg)^{1/p^\infty}$-almost map from $\Bm^\circ$ to $S^{pg}$. The proof is complete.
\end{proof}

We are ready to prove our main theorem in the unramified case.

\begin{thm}\label{main}
Let terminology be as in Notation~\ref{Notation}.
If $I$ is an ideal of $R$ and $J\subseteq R$, then $(pg)^{1/p^\infty}(IS:J)\subseteq (I:J)\Bm^\circ$ and consequently $(IS:J)\subseteq (I:J)S^{epf}$.
\end{thm}
\begin{proof}
Because there exists a chain of maps:
$$S\to B\to B\otimes \Ab \to \Bm^\circ,$$
we have $(IS:J)\subseteq (I\Bm^\circ:J)$ and the later one is $(pg)^{1/p^\infty}$-almost isomorphic to $(I:J)\Bm^\circ$ because $\Bm^\circ$ is $(pg)^{1/p^\infty}$-almost flat over $R$ by Lemma~\ref{Bmalmostflat}. Thus $(pg)^{1/p^\infty}(IS:J)\subseteq (I:J)\Bm^\circ$. Next, by Lemma \ref{almost map} there is a $(pg)^{1/p^\infty}$-almost map from $\Bm^\circ$ to $S^{pg}=R^{pg}$.
This implies that $(pg)^{1/p^\infty}(IS:J)\subseteq (I:J) S^{pg}$.

So for every $z\in (IS:J)$ and every $\epsilon=\frac{1}{p^t}$, $(pg)^\epsilon z\in (I:J)S^{pg}$, and thus $z\in (I:J)S^{epf}$ by Lemma~\ref{new complete}.
\end{proof}

Now we can prove the usual form of colon-capturing:

\begin{cor}\label{CC}
Let terminology be as in Notation~\ref{Notation}.
 Suppose $y_1,\dots,y_n$ are parameters in $S$.  Then $(pg)^{1/p^\infty}((y_1,\dots,y_{n-1})\Bm^\circ :y_n)\subseteq (y_1,\dots,y_{n-1})\Bm^\circ$. As a consequence, we have $((y_1,\ldots,y_{n-1})S:y_n)\subseteq (y_1,\ldots,y_{n-1})S^{epf}$.
\end{cor}
\begin{proof}
We first prove $(pg)^{1/p^\infty}((y_1,\dots,y_{n-1})\Bm^\circ :y_n)\subseteq (y_1,\dots,y_{n-1})\Bm^\circ$. Assume $k$ of the elements $y_i$ actually belong to $R$. If $k=n$, the claim is true because $\Bm^\circ$ is $(pg)^{1/p^\infty}$-almost flat over $R$ by Lemma~\ref{Bmalmostflat}.
We now complete the proof by induction on $n-k$; the case $n-k=0$ is done.
Suppose we have a counterexample with $n-k$ minimal.
First suppose $y_n\notin R$. Clearly $(y_1,\ldots,y_n)S\cap R$ is not contained in any minimal prime of $(y_1,\ldots,y_{n-1})S$; so there exists $w_n\in R$ such that $((y_1,\ldots,y_{n-1})\Bm^\circ:w_n)\supseteq((y_1,\ldots,y_{n-1})\Bm^\circ:y_n)$ and $y_1,\ldots,y_{n-1},w_n$ are parameters in $S$.
By the minimality assumption, we have
$(pg)^{1/p^\infty}((y_1,\ldots,y_{n-1})\Bm^\circ:y_n)\subseteq (pg)^{1/p^\infty}((y_1,\ldots,y_{n-1})\Bm^\circ:w_n)\subseteq (y_1,\ldots,y_{n-1})\Bm^\circ$ and so we did not have a counterexample.
Thus we know $y_n\in R$.
Without loss of generality, we may now assume $y_1\notin R$.
If $a_n\in ((y_1,\ldots,y_{n-1})\Bm^\circ:y_n)$, we have an equation $a_1y_1+\cdots+a_ny_n=0$ with each $a_i\in\Bm^\circ$.
Then $a_1\in ((y_2,\ldots,y_n)\Bm^\circ:y_1)$.
The argument we used above now shows $(pg)^{1/p^\infty}a_1\in (y_2,\ldots,y_n)\Bm^\circ$.
Consider any $\epsilon=1/p^t$.
We have $(pg)^\epsilon a_1=\sum_{i=2}^nb_iy_i$.
Then $\sum_{i=2}^n(b_iy_1+(pg)^\epsilon a_i)y_i=0$ and so $b_ny_1+(pg)^\epsilon a_n\in ((y_2,\ldots,y_{n-1})\Bm^\circ:y_n)$.
Since we have decreased $n$ but not $k$, $n-k$ has been decreased and so
$(pg)^\epsilon(b_ny_1+(pg)^\epsilon a_n)\in (y_2,\ldots,y_{n-1})\Bm^\circ$.
Thus $(pg)^{2\epsilon}a_n\in (y_1,\ldots,y_{n-1})\Bm^\circ$.
Hence $(pg)^{1/p^\infty}a_n\in (y_1,\ldots,y_{n-1})\Bm^\circ$. This finishes the proof of the first conclusion.

Finally, for every $z\in ((y_1,\dots,y_{n-1})S:y_n)$ and every $\epsilon=\frac{1}{p^t}$, $z\in ((y_1,\dots,y_{n-1})\Bm^\circ:y_n)$ and so we have $(pg)^{\epsilon}z\in (y_1,\ldots,y_{n-1})\Bm^\circ$. Now by Lemma~\ref{almost map}, $(pg)^{\epsilon}z\in (y_1,\ldots,y_{n-1})S^{pg}$ and thus $z\in (y_1,\ldots,y_{n-1})S^{epf}$ by Lemma~\ref{new complete}.
\end{proof}

Our next objective is to extend Theorem~\ref{main} to the ramified case.
\begin{thm}\label{new main}
Let $R'$ be a (possibly ramified) complete regular local ring of mixed characteristic $p$ {that has an F-finite residue field} and let the integral domain $S$ be a finite extension of $R'$.
If $I$ is an ideal of $R'$ and $J\subseteq R'$, then $(IS:J)\subseteq (I:J)S^{epf}$.
\end{thm}
\begin{proof}
$R'$ is a finite extension of an unramified complete local ring $R$. We choose $g$ so that $R[(pg)^{-1}]\to S[(pg)^{-1}]$  is finite \'{e}tale, and $(p,g)R$ is a height two ideal. We now use the standard framework of Notation~\ref{Notation}. We first note that, just as in the proof of Theorem~\ref{main}, our Lemma~\ref{almost map} and Lemma~\ref{new complete} will give us the desired result provided we can prove $(pg)^{1/p^\infty}(I\Bm^\circ:J)\subseteq (I:J)\Bm^\circ$. Thus, we will be done if we can show $\Bm^\circ$ is $(pg)^{1/p^\infty}$-almost flat over $R'$. This follows formally from the fact that $R'$ is regular and every system of parameters of $R'$ (in fact $S$) is a $(pg)^{1/p^\infty}$-almost regular sequence on $\Bm^\circ$ by Corollary~\ref{CC}. We give a detailed argument below.

\noindent\textbf{Claim.} $\Bm^\circ$ is $(pg)^{1/p^\infty}$-almost flat over $R'$.

\noindent{\textit{Proof of Claim.}} It is enough to show $\Tor_i^{R'}(N, \Bm^\circ)$ is $(pg)^{1/p^\infty}$-almost zero for all finitely generated $R'$-modules $N$ and every $i>0$. We use descending induction. This is clearly true when $i>\dim R'$ since $R'$ is regular. Now suppose $\Tor_{k+1}^{R'}(N, \Bm^\circ)$ is $(pg)^{1/p^\infty}$-almost zero for all finitely generated $N$. 
We want to show $\Tor_k^{R'}(N, \Bm^\circ)$ is $(pg)^{1/p^\infty}$-almost zero for all finitely generated $N$. By considering a prime cyclic filtration of $N$, it is enough to prove this for $N=R'/P$. Let $h=\Ht P$. We can pick a regular sequence $x_1,\dots,x_h\in P$. Now $P$ is an associated prime of $(x_1,\dots,x_h)$ and 
so we have $0\to R'/P\to R'/(x_1,\dots, x_h)\to C\to 0$. The long exact sequence for $\Tor$ gives: $$\Tor_{k+1}^{R'}(C, \Bm^\circ)\to \Tor_{k}^{R'}(R'/P, \Bm^\circ)\to \Tor_{k}^{R'}(R'/(x_1,\dots,x_h), \Bm^\circ).$$ Now  $\Tor_{k}^{R'}(R'/(x_1,\dots,x_h), \Bm^\circ)=H_k(x_1,\dots,x_h, \Bm^\circ)$ is $(pg)^{1/p^\infty}$-almost zero because $x_1,\dots,x_h$ is a $(pg)^{1/p^\infty}$-almost regular sequence on $\Bm^\circ$ by Corollary~\ref{CC}, and $\Tor_{k+1}^{R'}(C, \Bm^\circ)$ is $(pg)^{1/p^\infty}$-almost zero by induction. It follows that $\Tor_{k}^{R'}(R'/P, \Bm^\circ)$ is $(pg)^{1/p^\infty}$-almost zero.
\end{proof}

\begin{remark}
The analog of Theorem~\ref{main} and Theorem~\ref{new main} in equal characteristic $p>0$ is true if we replace extended plus closure by tight closure (in fact, even by the plus closure).
If $R\to S$ is a module-finite extension of excellent local domains of characteristic $p>0$ with $R$ regular and $I,J$ are ideals of $R$, then we have $(IS:JS)\subseteq (I:J)S^*$. To see this, let $z\in (IS:JS)$, then since we have $R\to S\to R^+$, we know that $z\in (IR^+:JR^+)$. Since $R$ is regular and $R^+$ is a balanced big Cohen-Macaulay algebra in characteristic $p>0$, $R^+$ is faithfully flat over $R$ and thus $(IR^+:JR^+)=(I:J)R^+$. Hence $z\in (I:J)R^+\cap S=(I:J)S^+\subseteq (I:J)S^*$.
\end{remark}

\begin{thm}\label{Tor}
Let $R'$ be a (possibly ramified) complete regular local ring of mixed characteristic $p$ {that has an F-finite residue field} and let the integral domain $S$ be a finite extension of $R'$.
Let $M$ be an $R'$-module. Then there exists $c\neq 0\in R'$ such that the natural map $c^{\epsilon}\Tor_i^{R'}(S,M)/p^N\to \Tor_i^{R'}(S^+,M)/p^N$ is zero for all $i\geq 1$ and every $N,\epsilon>0$. Consequently, for every $\alpha\in \Tor_i^{R'}({R'}^+,M)/p^N$, there exists $c\neq 0\in R'$ such that $c^{\epsilon}\alpha=0$ for every $\epsilon$.
\end{thm}
\begin{proof}
As before, $R'$ is a finite extension of an unramified complete regular local ring $R$. We choose $g$ so that $R[(pg)^{-1}]\to S[(pg)^{-1}]$  is finite \'{e}tale, and $(p,g)R$ is a height two ideal. We now use the standard framework of Notation~\ref{Notation}. By the Claim in the proof of Theorem~\ref{new main}, $\Bm^\circ$ is $(pg)^{1/p^\infty}$-almost flat over $R'$. Thus $(pg)^{1/p^\infty}\Tor_i^{R'}(\Bm^\circ, M)=0$ for all $i>0$. It follows that  $(pg)^{1/p^\infty}\Tor_i^{R'}(S, M)\to \Tor_i^{R'}(S^{pg},M)$ is zero since it factors through $\Tor_i^{R'}(\Bm^\circ,M)$ by Lemma~\ref{almost map}. At this point, we note that $S^{pg}$ is $(pg)^{1/p^\infty}$-almost isomorphic to $S^{++}$ by Lemma~\ref{new complete}; 
 thus the map $(pg)^{1/p^\infty}\Tor_i^{R'}(S,M)\to \Tor_i^{R'}(S^{++},M)$ is also zero. Finally, tensoring the commutative diagram
\[\xymatrix{
0 \ar[r] & S^+ \ar[r]^{\cdot p^N} \ar[d] & S^+ \ar[r]\ar[d] & S^+/p^N \ar[r] \ar[d]^= & 0 \\
0 \ar[r] & S^{++} \ar[r]^{\cdot p^N} & S^{++} \ar[r] & S^{++}/p^N \ar[r] & 0
}
\]
with $M$ induces a commutative diagram:
\[\xymatrix{
\Tor_i^{R'}(S^{+},M)/p^N \ar@{^{(}->}[r]\ar[d] & \Tor_i^{R'}(S^{+}/p^N,M) \ar[d]^= \\
\Tor_i^{R'}(S^{++},M)/p^N \ar@{^{(}->}[r] & \Tor_i^{R'}(S^{++}/p^N,M)
}
\]
By the injectivity of the map on the first line of the above diagram, in order to show the image of $(pg)^{1/p^\infty}\Tor_i^{R'}(S,M)$ is zero in $\Tor_i^{R'}(S^{+},M)/p^N$, it suffices to prove its image is zero in $\Tor_i^{R'}(S^{+}/p^N,M)$, which is clear because it factors through $\Tor_i^{R'}(S^{++},M)/p^N$ and we already know the image is zero in $\Tor_i^{R'}(S^{++},M)$.
\end{proof}

\begin{remark}
We cannot assert that $\Tor_i^{R'}({R'}^+,M)/p^N$ is almost zero or that ${R'}^+/p^N$ is almost flat because each $S$ requires its own $c$.
\end{remark}

Next we consider local cohomology.
We begin with a lemma.
This lemma follows from general results on derived completion (for example see \cite[Proposition 1.5]{GM}) and we believe it is known to experts.
However, as completions of modules which are not finite is a murky area to many of us, we offer a direct proof.

\begin{lemma}\label{Cohomlemma}
Suppose $T$ is any commutative ring, $f\in T$, $I=(f_1,\dots,f_n)$ is a finitely generated ideal containing a power of $f$. Let $M$ be a $T$-module such that $f$ is a nonzerodivisor on $M$. Then $H_I^i(M)\to H_I^i(\widehat{M})$ is an isomorphism.
\end{lemma}
\begin{proof}

We may assume that $f\in I$, say $f=f_n$. Note that $f$ is a nonzerodivisor on $\widehat{M}$.
To see this, suppose $fg=0$ in $\widehat{M}$ with $g\neq 0$.  We may write $g=g_1+f^ng_2$ where $g_1\in M-f^nM$ and $g_2\in \widehat{M}$.  Then the image of $fg_1$ in $\widehat{M}$ is in $f^{n+1}\widehat{M}$, so $fg_1\in f^{n+1}M$.
Then $g_1\in f^nM$ since $f$ is a nonzerodivisor on $M$, a contradiction.

By the relation between Koszul cohomology and local cohomology \cite[Theorem 3.5.6]{BH}, it suffices to prove the corresponding isomorphism for Koszul cohomology with respect to $f_1^t,\dots,f_n^t$ for all $t$. But since $f=f_n$ is a nonzerodivisor on both $M$ and $\widehat{M}$,
 by  \cite[Corollary 1.6.13]{BH}, we have
\begin{eqnarray*}
H^i(f_1^t,\dots,f_n^t; M)&\cong& H^{i-1}(f_1^t,\dots,f_{n-1}^t; M/f^tM)\\
&=&H^{i-1}(f_1^t,\dots,f_{n-1}^t; \widehat{M}/f^t\widehat{M})\\
&\cong& H^i(f_1^t,\dots,f_n^t; \widehat{M}).
\end{eqnarray*}
This finishes the proof.
\end{proof}

We can now give an affirmative answer to Question~\ref{Q}.

\begin{thm} \label{Cohom}
Let $(S,\n)$ be a complete local domain of mixed characteristic of dimension $d$ {that has an F-finite residue field}.
Then there exists $c\in S$ such that the natural map  $c^\epsilon H_{\n }^i(S)\to H_{\n }^i(S^+)$ is zero for all $i<d$ and all $\epsilon>0$.
More generally, if $I$ is any ideal of $S$ which contains a power of $p$, the natural map  $c^\epsilon H_I^i(S)\to H_I^i(S^+)$ is zero for all $i<h=\height I$ and all $\epsilon>0$.
\end{thm}

\begin{proof}
As usual, we may view $S$ as a finite module over a complete unramified regular local ring $(R,\m)$ with $\m=\n\cap R$.
We now can invoke Notation~\ref{Notation}.
As $\height (I\cap R)=h$ and $I$ contains $p^n$ for some $n$, by Theorem~\ref{andre} we can pick $y_1=p^n, y_2,\dots,y_h$ in $R$ such that $y_1,\dots,y_h$ is a $(pg)^{1/p^\infty}$-almost regular sequence on $\Bm^\circ$. Let $(y_1,\dots,y_h, z_1,\dots,z_s)$ be a generating set of $I$. Since $y_1,\dots,y_h$ is a $(pg)^{1/p^\infty}$-almost regular sequence on $\Bm^\circ$, the Koszul cohomology $H^i(\underline{y},\underline{z}, \Bm^\circ)$ is $(pg)^{1/p^\infty}$-almost zero for all $i <h$. Taking a direct limit, we find that $(pg)^{1/p^\infty}H_I^i(\Bm^\circ)=0$ for $i<h$.
By Lemmas \ref{almost map} and \ref{new complete}, we have a $(pg)^{1/p^\infty}$ almost map  $\Bm^\circ\to R^{++}$.
This gives a commutative diagram of ring maps and a corresponding diagram of local cohomology maps for any $\epsilon$:
\[\xymatrix{
 H_I^i(S) \ar[r] \ar[d] & H_I^i(S^+)\ar[r]^{\cdot(pg)^\epsilon} & H_I^i(S^+) \ar[d]^\alpha  \\
 H_I^i(\Bm^\circ) \ar[rr]^{\cdot(pg)^\epsilon}  &  {} & H_I^i(S^{++})
}
\]
As the composition $H_I^i(S)\to H_I^i(\Bm^\circ)\to H_I^i(S^{++})$ is the zero map, and $\alpha$ is an isomorphism by Lemma~\ref{Cohomlemma} applied to $T=M=R^+$ and $f=p$, the theorem holds with $c=pg$.
\end{proof}

Now we see that $H_\n^i(S^+)$ is almost zero in the same way that we got almost zero in Theorem~\ref{Tor}.

\begin{cor}
Let $(S,\n)$ be a complete local domain of mixed characteristic {that has an F-finite residue field}.
Suppose $i<n$ and $z\in H_\n^i(S^+)$.
Then there exists $c\in S$ (which depends on $z$) such that $c^\epsilon z=0$ for every $\epsilon>0$.
\end{cor}

\begin{proof}
This follows immediately from the preceding theorem because $S^+$ is just a direct union of finite extensions $S_j$ of $S$ and we can apply the preceding theorem to each $S_j$.
\end{proof}

The next theorem follows from Corollary~\ref{CC} and Theorem~\ref{4.4}.\footnote{ When the residue field of $R$ is not necessarily F-finite, we can take a faithfully flat extension $R\to R'$ such that $R'$ is a complete regular local ring with F-finite residue field.  Corollary~\ref{CC} and Theorem~\ref{4.4} then imply $(IR')^{epf}=IR'$ and so $I^{epf}\subseteq (IR')^{epf}\cap R=IR'\cap R=I$.}
However, as our proof relies heavily on Andr\'{e}'s very deep perfectoid Abhyankar lemma, we would like to offer a more direct proof based on our earlier work on the vanishing conjecture for maps of Tor.

\begin{thm}\label{RLR}
Let $R$ be a regular local ring of mixed characteristic and let $I$ be an ideal of $R$.
Then $I^{epf}=I$.
\end{thm}

\begin{proof}
We first claim that for every $c\neq 0$ in $R$, there exists $e$ such that the natural map $R\to R^+$ sending $1$ to $c^{1/p^e}$ is pure. We prove this by contradiction. Suppose there exists $c\neq 0$ such that for every $e$, the map $R\to R^+$ sending $1\to c^{1/p^e}$ is not pure. Let $x_1,\dots,x_d$ be a regular system of parameters of $R$. Since $R\xrightarrow{\cdot c^{1/p^e}} R^+$ is not pure, the map $H_\m^d(R)\xrightarrow{\cdot c^{1/p^e}} H_\m^d(R^+)$ sends the socle element of $H_\m^d(R)$ to $0$, i.e., there exists $t$ (depending on $e$) such that
\begin{equation}\label{c contained in colon}
c^{1/p^e}(x_1x_2\cdots x_d)^t\in (x_1^{t+1}, x_2^{t+1},\dots, x_d^{t+1})R^+.
\end{equation}
Next we claim that for every $t>0$,
\begin{equation}\label{lim closure contained in integral closure}
((x_1^{t+1}, x_2^{t+1},\dots, x_d^{t+1})R^+:(x_1x_2\cdots x_d)^t)\subseteq \overline{(x_1,\dots,x_d)R^+}.
\end{equation}
To see this, suppose $u(x_1x_2\cdots x_d)^t\in (x_1^{t+1}, x_2^{t+1},\dots, x_d^{t+1})R^+$ but $u$ is not in $\overline{(x_1,\dots,x_d)R^+}$. Then there exists a domain $S$ module-finite over $R$ such that $u(x_1x_2\cdots x_d)^t\in (x_1^{t+1}, x_2^{t+1},\dots, x_d^{t+1})S$. Say
$$u(x_1x_2\cdots x_d)^t=r_1x_1^{t+1}+\cdots +r_dx_d^{t+1}.$$
Then $(r_1,\dots,r_d, -u)$ represents a class in $$\Tor_1^R(\frac{R}{(x_1^{t+1},\dots,x_d^{t+1}, (x_1x_2\cdots x_d)^t)}, S).$$ Since $u\notin\overline{(x_1,\dots,x_d)R^+}$, $u\notin\overline{(x_1,\dots,x_d)S}$. By the valuation criterion for integral closure, we know that there exists a mixed characteristic DVR $V$ such that $u\notin (x_1,\dots,x_d)V$. This means $(r_1,\dots,r_d, -u)$ represents a nonzero class in $$\Tor_1^R(R/(x_1^{t+1},\dots,x_d^{t+1}, (x_1x_2\cdots x_d)^t), V).$$
But this contradicts the vanishing conjecture for maps of Tor in mixed characteristic \cite[Theorem 1.3]{HM} applied to $R\to S\to V$. Now (\ref{c contained in colon}) and (\ref{lim closure contained in integral closure}) shows that $c^{1/p^e}\in \overline{(x_1,\dots,x_d)R^+}$ for all $e>0$. But this is clearly impossible by computing valuations.

Finally, suppose $x\in I^{epf}$, then for some $c\neq 0$, $c^\epsilon x\in (I, p^N)R^+$ for every $N, \epsilon$. Since the map $R\to R^+$ sending $1\to c^{1/p^e}$ is pure for some $e$, we know that $x\in (I,p^N)R$ for every $N$. This implies $x\in I$ and hence $I^{epf}=I$.
\end{proof}

We also have completed a closure-based proof of the Brian\c con-Skoda Theorem.

\begin{thm} \label{L} \cite{LS}
Let $R$ be a regular ring and let $I$ be an ideal of $R$ generated by $n$ elements.
Then the integral closure of $I^{d+n}$ is contained in $I^{d+1}$.
\end{thm}
\begin{proof}
It suffices to prove the result locally and so we can assume that we have either an equicharacteristic regular local ring or a mixed characteristic regular local ring.
Hochster and Huneke have proved the result in the equal characteristic case.
See \cite [Theorem 5.4]{HH1} for equal characteristic $p$ case and \cite{HH2} for equal characteristic $0$ case (or \cite[section 9.6]{L} for generalizations in the equal characteristic $0$ case using multiplier ideals). We can now handle the mixed characteristic case.
By Theorem~\ref{BS}, we know that the integral closure of ${I^{d+n}}$ is contained in $(I^{d+1})^{epf}$.
Theorem~\ref{RLR} completes the proof.
\end{proof}

\end{document}